\documentclass[12pt,reqno,a4paper]{amsart}
\usepackage{euscript}

\newtheorem{theorem}{Theorem}

\newtheorem{proposition}[theorem]{Proposition}
\newtheorem{lemma}{Lemma}
\newtheorem{remark}{Remark}
\newtheorem{example}{Example}

\renewcommand{\epsilon}{\varepsilon}
\DeclareMathOperator{\Ima}{Im}
\DeclareMathOperator{\Ker}{Ker}

\def\Id{\text{\rm Id}}
\def\cA{\EuScript{A}}
\def\cB{\EuScript{B}}

\def\N{\mathbb{N}}
\def\Z{\mathbb{Z}}
\def\R{\mathbb{R}}
\begin{document}

\title{Admissibility and  nonuniform polynomial dichotomies}
\begin{abstract}
For a general one-sided nonautonomous dynamics defined by a sequence of linear
operators, we consider the notion of a polynomial dichotomy with respect to a sequence of norms
and we characterize it completely in terms of the admissibility of bounded solutions. As a nontrivial
application, we establish the robustness of the notion of a nonuniform polynomial dichotomy. 
\end{abstract}
\begin{thanks}
{ D.D. was supported by the  Croatian Science Foundation under the project IP-2014-09-2285 and by the University of Rijeka under the project uniri-prirod-18-9}
\end{thanks}

\author{Davor Dragi\v cevi\'c}
\address{Department of Mathematics, University of Rijeka, Croatia}
\email{ddragicevic@math.uniri.hr}
\maketitle

\section{Introduction}
The problem of characterizing hyperbolic behaviour of dynamical systems  in terms of the spectral properties of certain linear operators has a long history that goes back to the pioneering works of Perron~\cite{Pe}
and Li~\cite{tali}. More precisely, Perron~\cite{Pe} established a complete characterization of the exponential stability of a linear differential equation \[x'=A(t)x \] in $\R^n$ in terms of the solvability (in $x$) of the nonlinear equation
\begin{equation}\label{adm}
x'=A(t)x+f(t), 
\end{equation}
where $f$ and $x$ belong to suitable function spaces. Similar results for the discrete time dynamics were obtained by Li~\cite{tali}. The condition that~\eqref{adm} has a (unique) solution $x$ in some space $Y_1$ for any choice of $f$ that belongs to some (possibly different) space $Y_2$  is commonly referred to as \emph{admissibility} condition. We note that  this requirement can be formulated in terms of spectral properties of the  linear 
operator 
\[
(Lx)(t)=x'-A(t)x
\]
acting between suitable function spaces. 

A  fundamental  contribution to this line of the research is due to Massera and Sch\"affer~\cite{MS1, MS} (see also Coppel~\cite{Co}). Indeed, in a constrast to the work of Perron, they have established complete characterization (in terms of admissibility) of the notion of a  (uniform) exponential dichotomy which includes the notion of exponential stability as a very particular case. 
More precisely, rather than considering only the dynamics that exhibits stable behaviour, they have considered the case of dynamics with the property that the phase space splits into two complementary directions, where in one direction dynamics  exhibits stable behaviour while in the complementary direction  it possesses an  unstable (chaotic) behaviour. In addition, they have developed an axiomatic approach 
to the problem of constructing all possible pairs $(Y_1, Y_2)$  of function spaces with the property that the corresponding admissibility condition is equivalent to the existence of exponential dichotomy.   The related results in the context of smooth dynamics have been established by Mather~\cite{Mather} and by  Chicone and Swanson~\cite{CS}. To the best of our knowledge, the  first results that deal  with the case of  infinite-dimensional dynamics are due to  Dalec$'$ki\u{\i} and  Kre\u\i{n}~\cite{DK} in the case of continuous time and by Henry~\cite{He} for noninvertible dynamics with discrete time.  For a detailed overview of those developments, we refer to~\cite{CL}. 

For more recent relevant contributions   that deal with continuous or discrete evolution families, we refer to~\cite{AVM,HH, HVM, LRS,  Pituk, P, PPP, PPP1, SS2, ALS1,SS, SS1,Sx2}. Furthermore, for results devoted to linear cocycles over maps and flows, we refer to~\cite{Leiva,LS,PPC,SS3,SS4} and reference therein. We stress that all the above mentioned works deal with \emph{uniform} exponential behaviour. Related results which deal with various flavours of \emph{nonuniform} exponential behaviour can be found in~\cite{BDV1, BDV,BDV3,LP,MSS1,MP,ZLZ,ZZ, ZLZ2}. Finally, for a detailed survey devoted to this line of the research, we refer to~\cite{BDV2}.

We emphasize that all the works that we mentioned deal with \emph{exponential} dichotomies. While exponential behaviour certainly  has a privilaged role due to its presence in the context of smooth  dynamics, it is not the only possible form of the asymptotic behaviour.  To the best of our knowledge, the first ones to study dichotomies with non-exponential growth rates were Preda and 
Megan~\cite{MP}. 
 Subsequent results are due to Muldowney~\cite{M} and  Naulin and Pinto~\cite{NP, NP1}. A systematic study of nonuniform dichotomies with arbitratry growth rates was initiated by Barreira and Valls~\cite{BV}. An important class of those dichotomies are the so-called \emph{nonuniform polynomial} dichotomies introduced independently (and in a slightly different form)
by Barreira and Valls~\cite{BV1} and by Bento and Silva~\cite{BS1, BS2}. In addition, Barreira and Valls gave sufficient conditions (in terms of nonvanishing of the so-called polynomial Lyapunov exponents) for the existence of nonuniform polynomial dichotomies (see Remark~\ref{y9}).

The main objective of the present paper is to obtain a full characterization of the class of nonuniform polynomial dichotomies for dynamics with discrete time in terms of the appropriate admissibility property. This is achieved by studying the notion of a polynomial dichotomy with respect to a sequence of norms which includes the notion of a nonuniform polynomial dichotomy as a particular case. 
To the best of our knowledge, our results are the first one devoted to the characterization of nonuniform dichotomies with non-exponential growth rates via admissibility property (some much weaker results that deal only with contractions and expansions were obtained in~\cite{BDV4}). Our methods combine classical admissibility techniques (the so-called test sequence  method developed by Perron) together with recent contributions to the admissibility in relation with nonuniform exponential dichotomies~\cite{BDV1, BDV}. Furthermore, we build on the work of Hai~\cite{Hai} who obtained similar results for evolution families and considered \emph{uniform} polynomial contractions and expansions. We then apply our results to give a short proof of the robustness of nonuniform 
polynomial dichotomies. Although the robustness property of nonuniform polynomial dichotomies has been obtained earlier in~\cite{BV2}, we here present a much shorter proof. 

The paper is organized as follows. In Section~\ref{P} we introduce the notion of a polynomial dichotomy with respect to a sequence of norms. Then, in Section~\ref{MR} we obtain a complete characterization of this notion in terms of the appropriate admissibility property. Finally, in Sections~\ref{NPD} and~\ref{R} we apply our results to the study of nonuniform polynomial dichotomies.

\section{Preliminaries}\label{P}
Let $X=(X, \lVert \cdot \rVert)$ be a Banach space and let $B(X)$ denote the space of all bounded linear operators on $X$.  Moreover, let $\lVert \cdot \rVert_m$ for $m\in \N$ be a sequence of norms on $X$ such that  $\lVert \cdot \rVert_m$ is equivalent to $\lVert \cdot \rVert$ for each $m$.
Given a sequence $(A_m)_{m\in \N} \subset B(X)$, we define
\[
\cA(m, n)=\begin{cases}
A_{m-1} \cdots A_n & \text{if $m>n$,} \\
\Id & \text{if $m=n$.}
\end{cases}
\]
We say that $(A_m)_{m\in \N}$ admits a \emph{polynomial dichotomy} with respect to the sequence of norms $\lVert \cdot \rVert_m$
if:
\begin{itemize}
\item there exist projections $P_m$, $m\in \N$ satisfying
\begin{equation}\label{pro}
A_mP_m=P_{m+1}A_m, \quad m\in \N,
\end{equation}
such that each map $A_m \rvert_{\Ker P_m} \colon \Ker P_m \to \Ker P_{m+1}$ is invertible;
\item there exist $\lambda, D>0$ such that for every $x\in X$ and $m, n\in \N$ we have 
\begin{equation}\label{d1}
\lVert \cA(m, n)P_n x\rVert_m \le D(m/n)^{-\lambda} \lVert x\rVert_n \quad \text{for $m\ge n$}
\end{equation}
and 
\begin{equation}\label{d2}
\lVert \cA(m, n)Q_nx \rVert_m \le  D(n/m)^{-\lambda}\lVert x\rVert_n \quad \text{for $m\le n$,}
\end{equation}
where $Q_n=\Id-P_n$ and
\[
\cA(m, n)=(\cA(n, m)\rvert_{\Ker P_m})^{-1} \colon \Ker P_n \to \Ker P_m 
\]
for $m<n$.
\end{itemize}

We also introduce a class of sequence spaces that will play a major role in our arguments. 
Let $Y$ be the set of all sequence $\mathbf x=(x_m)_{m\in \N}\subset X$  such that
\[
\lVert \mathbf x\rVert_\infty:=\sup_{m\in \N}\lVert x_m\rVert_m <\infty.
\]
Clearly, $(Y, \lVert \cdot \rVert_\infty)$ is a Banach space. Furthermore, for a given closed subspace $Z\subset X$, let $Y_Z$ be the set of all $\mathbf x=(x_m)_{m\in \N}\in Y$ such that $x_1\in Z$. It is easy to verify that $Y_Z$ is a closed subspace of $Y$. In the particular case when $Z=\{0\}$, we write $Y_0$ instead of $Y_{\{0\}}$.

Let
\[
\mathcal D(T_Z):=\bigg{\{}\mathbf x=(x_m)_{m\in \N}\in Y_Z:  \sup_{m\in \N} \big{(}(m+1) \lVert x_{m+1}-A_m x_m\rVert_{m+1}\big{)} <\infty \bigg{\}}.
\]
Moreover, we consider the linear operator $T_Z \colon \mathcal D(T_Z)\to Y_0$ defined by 
\[
(T_Z \mathbf x)_1=0 \quad \text{and} \quad (T_Z \mathbf x)_{m+1}=(m+1)(x_{m+1}-A_m x_m) \quad  \text{for $m\in \N$.}
\]

\begin{proposition}\label{prop1}
We have that $T_Z$ is a closed linear operator. 
\end{proposition}
\begin{proof}
Let $(\mathbf x^k)_{k\in \N}$  be a sequence in $\mathcal D(T_Z)$ converging to $\mathbf x\in Y_Z$ such that $T_Z\mathbf x^k$ converges to $\mathbf y\in Y_0$. For each $m\in \N$, we have (recall that each $\lVert \cdot \rVert_m$ is equivalent to $\lVert \cdot \rVert$) that
\[
\begin{split}
(m+1)(x_{m+1}-A_m x_m)&= \lim_{k\to \infty} (m+1)(x_{m+1}^k-A_m x_m^k) \\
&=\lim_{k\to \infty} (T_Z \mathbf x^k)_{m+1} \\
&=y_{m+1}.
\end{split}
\]
Hence, $T_Z\mathbf x=\mathbf y$ and consequently $\mathbf x\in \mathcal D(T_Z)$. We conclude that $T_Z$ is a closed operator. 
\end{proof}
For $\mathbf x\in \mathcal D(T_Z)$, let 
\[
\lVert \mathbf x\rVert_{T_Z}:=\lVert \mathbf x\rVert_\infty+\lVert T_Z\mathbf x\rVert_\infty.
\]
It follows from Proposition~\ref{prop1} that $(\mathcal D(T_Z), \lVert \cdot \rVert_{T_Z})$ is a Banach space. Furthermore, the operator
\[
T_Z \colon (\mathcal D(T_Z), \lVert \cdot \rVert_{T_Z}) \to Y_0
\]
is bounded and from now on we will denote it simply by $T_Z$.

\section{Main results}\label{MR}
In this section we present the main results of our paper. More precisely, we show that the notion of a polynomial dichotomy with respect to a sequence of norms can be characterized in terms of spectral properties of operators $T_Z$.
\begin{theorem}\label{t2}
If the sequence $(A_m)_{m\in \N}$ admits a polynomial dichotomy with respect to the sequence of norms $\lVert \cdot \rVert_m$,  then for the closed subspace  $Z=\Ima Q_1$ the operator $T_Z$ is invertible. 
\end{theorem}

\begin{proof}
Let $Z=\Ima Q_1$ and let us show that $T_Z$ is an invertible operator. We begin by showing that $T_Z$ is surjective.  Observe that without any loss of generality we can assume that~\eqref{d1} and~\eqref{d2} hold with $\lambda \in (0, 1)$.
 Choose $\mathbf y=(y_n)_{n\in \N}\in Y_0$ and define a sequence $\mathbf x=(x_n)_{n\in \N}\subset X$ by 
\begin{equation}\label{xn}
x_n=\sum_{k=1}^n \frac 1 k \cA(n, k)P_k y_k-\sum_{k=n+1}^\infty \frac{1}{k}\cA(n, k)Q_k y_k, \quad n\in \N.
\end{equation}
Observe that it follows from~\eqref{d1} that 
\begin{align}\label{438}
\begin{split}
\bigg{\lVert} \sum_{k=1}^n \frac 1 k \cA(n, k)P_k y_k \bigg{\rVert}_n &\le \sum_{k=1}^n \frac{1}{k}\lVert \cA(n,k)P_ky_k\rVert_n \\
&\le D\sum_{k=1}^n \frac 1 k (n/k)^{-\lambda}\lVert y_k\rVert_k \\
&\le Dn^{-\lambda}\lVert \mathbf y\rVert_\infty \sum_{k=1}^n \frac{1}{k^{-\lambda+1}} \\
&\le D n^{-\lambda}\lVert \mathbf y\rVert_\infty \bigg{(}1+\int_1^n t^{\lambda-1}\, dt \bigg{)} \\
&=D n^{-\lambda}\lVert \mathbf y\rVert_\infty (1+n^\lambda/\lambda-1/\lambda), 
\end{split}
\end{align}
for each $n\in \N$ and thus 
\begin{equation}\label{137}
\sup_{n\in \N}\bigg{\lVert} \sum_{k=1}^n \frac 1 k \cA(n, k)P_k y_k \bigg{\rVert}_n <\infty. 
\end{equation}
Similarly, \eqref{d2} implies that 
\[
\begin{split}
\bigg{\lVert}\sum_{k=n+1}^\infty \frac{1}{k}\cA(n, k)Q_k y_k \bigg{\rVert}_n &\le \sum_{k=n+1}^\infty \frac 1 k \lVert \cA(n, k)Q_k y_k\rVert_n \\
&\le D\sum_{k=n+1}^\infty \frac 1 k (k/n)^{-\lambda} \lVert y_k\rVert_k \\
&\le Dn^\lambda \lVert \mathbf y\rVert_\infty \sum_{k=n+1}^\infty \frac{1}{k^{\lambda+1}}\\
&\le Dn^\lambda \lVert \mathbf y\rVert_\infty \int_n^\infty t^{-\lambda -1}\, dt \\
&\le \frac{D}{\lambda}  \lVert \mathbf y\rVert_\infty .
\end{split}
\]
We conclude that
\begin{equation}\label{138}
\sup_{n\in \N}\bigg{\lVert}\sum_{k=n+1}^\infty \frac{1}{k}\cA(n, k)Q_k y_k \bigg{\rVert}_n <\infty.
\end{equation}
By~\eqref{137} and~\eqref{138}, we have that $\mathbf x\in Y$. Since $y_1=0$, it follows from~\eqref{xn} that $x_1\in Z$ and thus $\mathbf x\in Y_Z$. Furthermore, we have that
\[
\begin{split}
x_{n+1}-A_nx_n &=\sum_{k=1}^{n+1} \frac 1 k \cA(n+1, k)P_k y_k-A_n\sum_{k=1}^n \frac 1 k \cA(n, k)P_k y_k \\
&\phantom{=}-\sum_{k=n+2}^\infty \frac{1}{k}\cA(n+1, k)Q_k y_k+A_n\sum_{k=n+1}^\infty \frac{1}{k}\cA(n, k)Q_k y_k \\
&=\sum_{k=1}^{n+1} \frac 1 k \cA(n+1, k)P_k y_k-\sum_{k=1}^n \frac 1 k \cA(n+1, k)P_k y_k \\
&\phantom{=}-\sum_{k=n+2}^\infty \frac{1}{k}\cA(n+1, k)Q_k y_k+\sum_{k=n+1}^\infty \frac{1}{k}\cA(n+1, k)Q_k y_k \\
&=\frac{1}{n+1}P_{n+1}y_{n+1}+\frac{1}{n+1}Q_{n+1}y_{n+1}\\
&=\frac{1}{n+1}y_{n+1},
\end{split}
\]
for each $n\in \N$ and thus $T_Z\mathbf x=\mathbf y$.

Let us now prove that $T_Z$ is injective. Take $\mathbf x=(x_n)_{n\in \N}\in Y_Z$ such that $T_Z\mathbf x=0$. We have that $x_n=\cA(n, 1)x_1$ for each $n\in \N$. Therefore, it follows from~\eqref{d2} that 
\[
\lVert x_1\rVert_1=\lVert \cA(1, n)x_n\rVert_1 \le Dn^{-\lambda}\lVert x_n\rVert_n \le Dn^{-\lambda}\lVert \mathbf x\rVert_\infty,
\]
for every $n\in \N$. By passing to the limit when $n\to \infty$, we conclude that $x_1=0$ and thus $x_n=0$ for each $n\in \N$. We conclude that $\mathbf x=0$. 
\end{proof}

\begin{theorem}\label{t1}
Assume that there exists a closed subspace $Z\subset X$ such that $T_Z$ is an invertible operator. Furthermore, suppose that there exist $M, a>0$ such that
\begin{equation}\label{pb}
\lVert \cA(m, n)x\rVert_m \le M (m/n)^a \lVert x\rVert_n \quad \text{for $m\ge n$ and $x\in X$.}
\end{equation}
 Then,  $(A_m)_{m\in \N}$ admits a polynomial dichotomy with respect to the sequence of norms $\lVert \cdot \rVert_m$. 
\end{theorem}

\begin{proof}
For each $n\in \N$, let
\[
X(n)=\bigg{\{}x\in X: \sup_{m\ge n}\lVert \cA(m, n)x\rVert_m <\infty \bigg{\}} \quad \text{and} \quad Z(n)=\cA(n, 1)Z.
\]
Note that $X(n)$ and $Z(n)$ are subspaces of $X$ for each $n\in \N$.
\begin{lemma}
For $n\in \N$, we have 
\begin{equation}\label{oplus}
X=X(n)\oplus Z(n).
\end{equation}
\end{lemma}

\begin{proof}[Proof of the lemma]
Let us first consider the case when $n\ge 2$. Take $v\in X$ and define  a sequence $\mathbf y=(y_m)_{m\in \N}\subset X$ by 
\[
y_m=\begin{cases}
v & \text{if $m=n$,} \\
0 & \text{if $m\neq n$.}
\end{cases}
\]
Obviously $\mathbf y\in Y_0$. Since $T_Z$ is invertible, there exists $\mathbf x=(x_m)_{m\in \N}\in Y_Z$ such that $T_Z\mathbf x=\mathbf y$. Hence, 
\[
x_m-A_{m-1}x_{m-1}=\begin{cases}
\frac 1 n v & \text{if $m=n$,} \\
0 & \text{if $m\neq n$.}
\end{cases}
\]
In particular, 
\[
\sup_{m\ge n}\lVert \cA(m, n)nx_n\rVert_m =n\sup_{m\ge n} \lVert x_m\rVert_m\le n\lVert \mathbf x\rVert_\infty <\infty, 
\]
and thus $nx_n \in X(n)$. On the other hand, 
\[
-nA_{n-1}x_{n-1}=\cA(n, 1)(-nx_1)\in \cA(n, 1)Z=Z(n).
\]
Since $v=nx_n-nA_{n-1}x_{n-1}$, we conclude that $v\in X(n)+Z(n)$.

Take now $v\in X(n)\cap Z(n)$. Choose $w\in Z$ such that $v=\cA(n, 1)w$. Let us define a sequence $\mathbf x=(x_m)_{m\in \N}$ by
\[
x_m=\cA(m, 1)w, \quad m\in \N.
\]
Observe that $\mathbf x\in Y_Z$ and $T_Z\mathbf x=0$. Since $T_Z$ is injective we have that $\mathbf x=0$ and consequently $v=0$.  We conclude that~\eqref{oplus} holds.

Let us now establish~\eqref{oplus} for $n=1$. Take $v\in X$ and define the sequences
\[
\mathbf x^1=(v, 0, 0, \ldots) \quad \text{and} \quad \mathbf y^1=(0, -2A_1v, 0, 0, \ldots ).
\]
Note that 
\[
x_{m+1}^1-A_m x_m^1=\frac{1}{m+1}y_{m+1}^1 \quad \text{for $m\in \N$.}
\]
On the other hand, since $\mathbf y^1\in Y_0$, there exists $\mathbf x^2=(x_m^2)_{m\in \N}\in Y_Z$  such that $T_Z\mathbf x^2=\mathbf y^1$. Consequently, 
\[
x_m^1-x_m^2=\cA(m, 1)(v-x_1^2)
\]
for $m\in \N$. Since $\mathbf x^1-\mathbf x^2\in Y$, we have that $v-x_1^2\in X(1)$. On the other hand, since $\mathbf x^2\in Y_Z$ we have that $x_1^2\in Z$ and thus 
\[
v=v-x_1^2+x_1^2\in X(1)+Z=X(1)+Z(1).
\]
Take now $v\in X(1)\cap Z(1)$ and define  a sequence $\mathbf x=(x_m)_{m\in \N}$ by 
\[
x_m=\cA(m, 1)v, \quad m\in \N.
\]
Clearly, $\mathbf x\in Y_Z$ and $T_Z\mathbf x=0$. Hence, $\mathbf x=0$ and thus $v=0$. We have proved that~\eqref{oplus} holds for $n=1$ also.  Hence, the proof of the lemma is completed. 
\end{proof}
Observe that
\begin{equation}\label{direct}
A_nX(n)\subset X(n+1) \quad \text{and} \quad A_nZ(n)= Z(n+1),
\end{equation}
for each $n\in \N$.
\begin{lemma}
For each $n\in \N$,
\[
A_n \rvert_{Z(n)} \colon Z(n) \to Z(n+1)
\]
is an invertible linear map. 
\end{lemma}

\begin{proof}[Proof of the lemma]
Clearly, $A_n \rvert_{Z(n)}$ is surjective.  
Assume now that there exists $v\in Z(n)$ such that $A_n v=0$.  Hence, $\cA(m, n)v=0$ for $m>n$ and therefore $v\in X(n)$. Therefore, it follows from~\eqref{oplus} that
$v\in X(n)\cap Z(n)=\{0\}$. 

\end{proof}
We now show that vectors in $X(n)$ exhibit uniform polynomial decay under the action of the cocycle $\cA$. 
\begin{lemma}
There exist $D, \lambda >0$ such that
\begin{equation}\label{201}
\lVert \cA(m, n)x\rVert_m \le D(m/n)^{-\lambda}\lVert x\rVert_n, \quad \text{for $m\ge n$ and $x\in X(n)$.}
\end{equation}
\end{lemma}

\begin{proof}[Proof of the lemma]
We first claim that there exists $L>0$ such that 
\begin{equation}\label{o}
\lVert \cA(m, n)x\rVert_m \le L\lVert x\rVert_n \quad \text{for $m\ge n$ and $x\in X(n)$.}
\end{equation}
Let us first consider the case when $m\ge 2n$ and take $x\in X(n)$ such that $\cA(m, n)x\neq 0$.  Consequently, $\cA(k,n)x\neq 0$ for $n+1\le k \le m$. Let us consider sequences $\mathbf y=(y_k)_{k\in \N}$ and $\mathbf x=(x_k)_{k\in \N}$ defined by 
\[
y_k=\begin{cases}
0 & \text{if $1\le k \le n$,} \\
\frac{\cA(k, n)x}{\lVert \cA(k, n)x\rVert_k} & \text{if $n+1\le k \le m$,}\\
0 & \text{if $k>m$,}
\end{cases}
\]
and 
\[
x_k=\begin{cases}
0 & \text{if $1\le k \le n$,} \\
\sum_{j=n+1}^k \frac{\cA(k, n)x}{j\lVert \cA(j, n)x\rVert_j} & \text{if $n+1\le k \le m$,}\\
\sum_{j=n+1}^m \frac{\cA(k, n)x}{j\lVert \cA(j, n)x\rVert_j} & \text{if $k>m$.}
\end{cases}
\]
Note that $\mathbf y\in Y_0$. Furthermore, since $x\in X(n)$ we have that $\mathbf x\in Y_Z$. It is  straightforward to verify that $T_Z\mathbf x=\mathbf y$. Consequently,
\[
\lVert \mathbf x\rVert_\infty \le \lVert \mathbf x\rVert_{T_Z}=\lVert T_Z^{-1}\mathbf y\rVert_{T_Z}\le \lVert T_Z^{-1}\rVert \cdot \lVert \mathbf y\rVert_\infty \le  \lVert T_Z^{-1}\rVert.
\]
Therefore, 
\begin{equation}\label{1245}
\lVert T_Z^{-1}\rVert \ge \lVert \mathbf x\rVert_\infty \ge \lVert x_m\rVert_m = \lVert \cA(m, n)x\rVert_m \sum_{j=n+1}^m \frac{1}{j\lVert \cA(j, n)x\rVert_j},
\end{equation}
and thus 
\[
\lVert T_Z^{-1}\rVert \ge \lVert \cA(m, n)x\rVert_m \sum_{j=n+1}^{2n} \frac{1}{j\lVert \cA(j, n)x\rVert_j}.
\]
On the other hand, \eqref{pb} implies that
\[
\lVert \cA(j, n)x\rVert_j \le M(j/n)^a \lVert x\rVert_n \le M2^a\lVert x\rVert_n, \quad \text{for $n< j \le 2n$.}
\]
Hence, 
\[
\lVert T_Z^{-1}\rVert \ge \frac{\lVert \cA(m, n)x\rVert_m}{M2^a\lVert x\rVert_n} \sum_{j=n+1}^{2n} \frac 1 j \ge  \frac{\lVert \cA(m, n)x\rVert_m}{M2^{a+1}\lVert x\rVert_n},
\]
which yields 
\begin{equation}\label{x1}
\lVert \cA(m, n)x\rVert_m \le M2^{a+1}\lVert T_Z^{-1}\rVert \cdot \lVert x\rVert_n.
\end{equation}
Moreover, \eqref{pb} implies that 
\begin{equation}\label{x2}
\lVert \cA(m, n)x\rVert_m  \le M2^a \lVert x\rVert_n \quad \text{for $n\le m\le 2n$ and $x\in X$.}
\end{equation}
By~\eqref{x1} and~\eqref{x2}, we conclude that~\eqref{o} holds with
\[
L:=\max \{M2^a, M2^{a+1}\lVert T_Z^{-1}\rVert \}>0. 
\]
We next show that there exists $N_0\in \N$  such that
\begin{equation}\label{1257}
\lVert \cA(m, n)x\rVert_m \le e^{-1} \lVert x\rVert_n \quad \text{for $m\ge N_0n$ and $x\in X(n)$.}
\end{equation}
Since 
\begin{equation}\label{hs}
\sum_{j=n+1}^{N_0n}\frac{1}{j}\ge \log (N_0n+1)-\log n-1 \ge \log N_0-1,
\end{equation}
we have (using~\eqref{o} and~\eqref{1245}) that 
\[
\begin{split}
\lVert T_Z^{-1}\rVert &\ge \lVert \cA(m, n)x\rVert_m \sum_{j=n+1}^m \frac{1}{j\lVert \cA(j, n)x\rVert_j} \\
&\ge \lVert \cA(m, n)x\rVert_m \sum_{j=n+1}^m \frac{1}{jL\lVert x\rVert_n} \\
&\ge \lVert \cA(m, n)x\rVert_m \sum_{j=n+1}^{N_0n} \frac{1}{jL\lVert x\rVert_n}  \\
&\ge  \frac{(\log N_0-1)\lVert \cA(m, n)x\rVert_m}{L\lVert x\rVert_n},
\end{split}
\]
and thus
\[
\lVert \cA(m, n)x\rVert_m \le \frac{L\lVert T_Z^{-1}\rVert}{\log N_0-1}\lVert x\rVert_n.
\]
Hence, if choose $N_0$ large enough so that
\[
\frac{L\lVert T_Z^{-1}\rVert}{\log N_0-1}\le e^{-1},
\]
we conclude that~\eqref{1257} holds. 

Take now arbitrary $m\ge n$, $x\in X(n)$  and choose largest $l\in \N\cup \{0\}$ such that $N_0^l \le m/n$.  It follows from~\eqref{o} and~\eqref{1257} that 
\[
\lVert \cA(m, n)x\rVert_m =\lVert \cA(m, N_0^ln)\cA(N_0^l n, n)x\rVert_m \le Le^{-l}\lVert x\rVert_n.
\]
Since $m/n <N_0^{l+1}$, we have that
\[
l> \frac{\log m/n}{\log N_0}-1,
\]
and thus
\[
e^{-l}\le e(m/n)^{-1/\log N_0}.
\]
Consequently, 
\[
\lVert \cA(m, n)x\rVert_m  \le Le (m/n)^{-1/\log N_0} \lVert x\rVert_n, 
\]
and we conclude that~\eqref{201} holds with
\[
D=Le \quad \text{and} \quad \lambda=1/\log N_0.
\]
The proof of the lemma is completed. 
\end{proof}
Next we show that nonzero vectors in $Z(n)$ exhibit uniform polynomial expansion under the action of the cocycle $\cA$.
\begin{lemma}
There exist $D, \lambda >0$ such that
\begin{equation}\label{202}
\lVert \cA(m, n)x\rVert_m \ge D (n/m)^{-\lambda}\lVert x\rVert_n, \quad \text{for $m\le n$ and $x\in Z(n)$.}
\end{equation}
\end{lemma}

\begin{proof}[Proof of the lemma]
Take $z\in Z\setminus \{0\}$ and $n> 2$. We consider sequences $\mathbf y=(y_k)_{k\in \N}$ and $\mathbf x=(x_k)_{k\in \N}$ defined by
\[
y_k=\begin{cases}
0 & \text{if $k=1$;} \\
-\frac{\cA(k, 1)z}{\lVert \cA(k, 1)z\rVert_k} & \text{if $2\le k \le n$;}\\
0 & \text{if $k>n$,}
\end{cases} 
\]
and 
\[
x_k=\begin{cases}
\sum_{j=k+1}^n \frac{\cA(k, 1)z}{j\lVert \cA(j, 1)z\rVert_j} & \text{if $1\le k\le n-1$;}\\
0 & \text{if $k\ge n$.}
\end{cases}
\]
Observe that $\mathbf y\in Y_0$ and $\mathbf x\in Y_Z$ (note that $x_1=az$ for $a\in \mathbb R$). Furthermore, it is straightforward to verify that $T_Z\mathbf x=\mathbf y$. Hence,
\[
\lVert \mathbf x\rVert_\infty \le \lVert \mathbf x\rVert_{T_Z}=\lVert T_Z^{-1}\mathbf y\rVert_{T_Z}\le \lVert T_Z^{-1}\rVert \cdot \lVert \mathbf y\rVert_\infty \le \lVert T_Z^{-1}\rVert.
\]
Therefore, for each $1\le k\le n-1$, we have that
\[
\lVert T_Z^{-1}\rVert \ge \lVert \cA(k, 1)z\rVert_k \sum_{j=k+1}^n \frac{1}{j\lVert \cA(j, 1)z\rVert_j}.
\]
Letting $n\to \infty$, we conclude that 
\begin{equation}\label{re}
\lVert T_Z^{-1}\rVert \ge \lVert \cA(k, 1)z\rVert_k \sum_{j=k+1}^\infty \frac{1}{j\lVert \cA(j, 1)z\rVert_j},
\end{equation}
for each $k\in \N$ and $z\in Z\setminus \{0\}$.
We now claim that there exists $L>0$ such that 
\begin{equation}\label{215}
\lVert \cA(m, 1)z\rVert_m \ge L\lVert \cA(n, 1)z\rVert_n \quad \text{for $m\ge n$ and $z\in Z$.}
\end{equation}
Using~\eqref{pb} and~\eqref{re}, we have that 
\[
\begin{split}
\frac{1}{\lVert \cA(n,1)z\rVert_n} &\ge \frac{1}{\lVert T_Z^{-1}\rVert }\sum_{j=n+1}^\infty \frac{1}{j\lVert \cA(j, 1)z\rVert_j} \\
&\ge \frac{1}{\lVert T_Z^{-1}\rVert }\sum_{j=m+1}^{2m} \frac{1}{j\lVert \cA(j, 1)z\rVert_j} \\
&= \frac{1}{\lVert T_Z^{-1}\rVert }\sum_{j=m+1}^{2m} \frac{1}{j\lVert \cA(j, m)\cA(m, 1)z\rVert_j} \\
&\ge   \frac{1}{\lVert T_Z^{-1}\rVert }\sum_{j=m+1}^{2m} \frac{1}{j M(j/m)^a \lVert \cA(m, 1)z\rVert_m} \\
&\ge  \frac{1}{M2^a \lVert T_Z^{-1}\rVert \cdot \lVert \cA(m, 1)z\rVert_m }\sum_{j=m+1}^{2m}  \frac 1 j \\
&\ge \frac{1}{M2^{a+1} \lVert T_Z^{-1}\rVert \cdot \lVert \cA(m, 1)z\rVert_m},
\end{split}
\]
which readily implies that~\eqref{215} holds with 
\[
L=\frac{1}{M2^{a+1} \lVert T_Z^{-1}\rVert}. 
\]
We next claim that there exists $N_0\in \N$ such that 
\begin{equation}\label{239}
\lVert \cA(m, 1)z\rVert_m \ge e\lVert \cA(n, 1)z\rVert_n \quad \text{for $m\ge N_0n$ and $z\in Z$.}
\end{equation}
Indeed, it follows from~\eqref{re} and~\eqref{215} that 
\[
\begin{split}
\frac{1}{\lVert \cA(n,1)z\rVert_n} &\ge \frac{1}{\lVert T_Z^{-1}\rVert }\sum_{j=n+1}^\infty \frac{1}{j\lVert \cA(j, 1)z\rVert_j} \\
&\ge \frac{1}{\lVert T_Z^{-1}\rVert }\sum_{j=n+1}^{N_0n} \frac{1}{j\lVert \cA(j, 1)z\rVert_j} \\
&\ge \frac{L}{\lVert T_Z^{-1}\rVert \cdot \lVert \cA(m, 1)z\rVert_m}\sum_{j=n+1}^{N_0n} \frac 1 j \\
&\ge \frac{L(\log N_0-1)}{\lVert T_Z^{-1}\rVert \cdot \lVert \cA(m, 1)z\rVert_m},
\end{split}
\]
where in the last step we have used~\eqref{hs}. Consequently, if we choose $N_0$ such that
\[
\frac{L(\log N_0-1)}{\lVert T_Z^{-1}\rVert}\ge e, 
\]
we have that~\eqref{239} holds. Proceeding as in the proof of the previous lemma, one can easily conclude that there exist $D, \lambda>0$ such that
\begin{equation}\label{m}
\lVert \cA(m, 1)z\rVert_m \ge \frac 1 D (m/n)^\lambda \lVert \cA(n, 1)z\rVert_n \quad \text{for $m\ge n$ and $z\in Z$.}
\end{equation}
Take now $m\ge n$, $v\in Z(n)$ and choose $z\in Z$ such that $v=\cA(n, 1)z$. Hence, \eqref{m} gives that 
\[
\begin{split}
\lVert\cA(m, n)v\rVert_m &=\lVert \cA(m, 1)z\rVert_m  \\
&\ge \frac 1 D (m/n)^\lambda \lVert \cA(n, 1)z\rVert_n \\
&= \frac 1 D (m/n)^\lambda \lVert v\rVert_n,
\end{split}
\]
which readily implies that~\eqref{202} holds.
\end{proof}
By taking the maximum over the  two values of $D$ obtained in the previous two lemmas, we can assume that~\eqref{201} and~\eqref{202} hold with the same $D>0$.  Similarly, by taking the minimum of the two obtained values,  we can also  assume that those bounds hold with the same $\lambda$.

Let $P_n \colon X\to X(n)$ be the projection associated with the decomposition~\eqref{oplus} for $n\in \N$.  
\begin{lemma}
We have that~\eqref{pro} holds. 
\end{lemma}
\begin{proof}[Proof of the lemma]
Take an arbitrary $m\in \N$ and $x\in X$. Furthermore, write $x$ as $x=x_1+x_2$ with $x_1\in X(m)$ and $x_2\in Z(m)$. Then, $A_mx=A_mx_1+A_mx_2$ and  it follows from~\eqref{direct} that $A_m x_1\in X(m+1)$ and $A_m x_2\in Z(m+1)$. Hence, 
\[
P_{m+1}A_mx=A_mx_1=A_mP_m x,
\]
which yields the desired conclusion. 
\end{proof}

The final ingredient of the proof is the following lemma. 
\begin{lemma}
We have that
\begin{equation}\label{203}
\sup_{n\in \N}\lVert P_n\rVert <\infty. 
\end{equation}
\end{lemma}

\begin{proof}[Proof of the lemma]
For each $n\in \N$, let
\[
\gamma_n:=\inf \{\lVert v^s+v^u\rVert_n: \lVert v^s\rVert_n=\lVert v^u\rVert_n=1, \ v^s\in X(n), \ v^u \in Z(n)\}.
\]
Then (see~\cite[Lemma 4.2]{Sx2}), 
\begin{equation}\label{Pn}
\lVert P_n\rVert \le \frac{2}{\gamma_n}.
\end{equation}
Let us fix $v^s\in X(n)$  and $v^u \in Z(n)$ such that $\lVert v^s\rVert_n=\lVert v^u\rVert_n=1$. It follows from~\eqref{pb} that for all $m\ge n$,
\[
\lVert \cA(m, n)(v^s+v^u)\rVert_m \le M(m/n)^a \lVert v^s+v^u\rVert_n,
\]
and thus by~\eqref{201} and~\eqref{202} we have that 
\begin{align}\label{iu}
\begin{split}
\lVert v^s+v^u\rVert_n &\ge \frac{1}{M(m/n)^a} \lVert \cA(m, n)(v^s+v^u)\rVert_m \\
&\ge \frac{1}{M(m/n)^a} \bigg{(}\lVert \cA(m, n)v^u\rVert_m-\lVert \cA(m, n)v^s\rVert_m \bigg{)} \\
&\ge \frac{1}{M(m/n)^a} \bigg{(}\frac 1 D (m/n)^\lambda -D(m/n)^{-\lambda} \bigg{)}. \\
\end{split}
\end{align}
Choose now $N_0\in \N$ such that 
\[
\frac 1 DN_0^\lambda -DN_0^{-\lambda}>0. 
\]
Hence, it follows from~\eqref{iu} (by taking $m=N_0n$) that 
\[
\lVert v^s+v^u\rVert_n \ge \frac{1}{MN_0^a} \bigg{(}\frac 1 DN_0^\lambda -DN_0^{-\lambda} \bigg{)}=:c>0.
\]
Therefore, $\gamma_n \ge c$ and thus the conclusion of the lemma follows readily from~\eqref{Pn}.
\end{proof}
The conclusion of the theorem now follows directly from~\eqref{201}, \eqref{202} and~\eqref{203}.
\end{proof}

Let us now discuss Theorems~\ref{t2} and~\ref{t1} in the particular case of polynomial contractions and expansions.  We say that the sequence $(A_m)_{m\in \N}$ admits a \emph{polynomial contraction} with respect to the sequence of norms $\lVert \cdot \rVert_m$ if it admits a polynomial dichotomy with respect to the sequence of norms $\lVert \cdot \rVert_m$ and with
projections $P_m=\Id$, $m\in \N$.

Similarly, we say that the sequence $(A_m)_{m\in \N}$ admits a \emph{polynomial expansion} with respect to the sequence of norms $\lVert \cdot \rVert_m$ if it admits a polynomial dichotomy with respect to the sequence of norms $\lVert \cdot \rVert_m$ and with
projections $P_m=0$, $m\in \N$.

The following two results are direct consequences of Theorems~\ref{t2} and~\ref{t1}.
\begin{theorem}
The following two statements are equivalent:
\begin{itemize}
\item the sequence $(A_m)_{m\in \N}$ admits a polynomial contraction with respect to the sequence of norms $\lVert \cdot \rVert_m$;
\item there exist $M, a>0$ such that~\eqref{pb} holds and for each $\mathbf y=(y_n)_{n\in \N}\in Y_0$, the sequence $\mathbf x=(x_n)_{n\in \N}$ defined by
\begin{equation}\label{1:55}
x_n=\sum_{k=1}^n \frac 1 k\cA(n, k) y_k \quad n\in \N, 
\end{equation}
belongs to $Y_0$. 
\end{itemize}
\end{theorem}

\begin{proof}
Assume that the first statement holds. Obviously, \eqref{d1} (with $P_n=\Id$) implies that~\eqref{pb} holds with $M=D$ and any $a>0$. Furthermore, 
by proceeding as in~\eqref{438}, it is easy to show  for each $\mathbf y=(y_n)_{n\in \N}\in Y_0$, the sequence $\mathbf x=(x_n)_{n\in \N}$ defined by~\eqref{1:55} belongs to $Y_0$.

  Conversely, under the assumption that the second statement is valid, we have that $T_Z$ is invertible for $Z=\{0\}$. Then, Theorem~\ref{t1} implies that  the sequence $(A_m)_{m\in \N}$ admits a polynomial contraction with respect to the sequence of norms $\lVert \cdot \rVert_m$.
\end{proof}

\begin{theorem}
Assume that there exist $M, a>0$ such that~\eqref{pb} holds. Then, 
the following two statements are equivalent:
\begin{itemize}
\item the sequence $(A_m)_{m\in \N}$ admits a polynomial expansion with respect to the sequence of norms $\lVert \cdot \rVert_m$;
\item $T_Z$ is an invertible operator for $Z=X$.
\end{itemize}
\end{theorem}

\begin{proof}
The conclusion of the theorem follows directly from Theorems~\ref{t2} and~\ref{t1}.
\end{proof}

We stress that it was proved in~\cite{BDV} that the version of Theorem~\ref{t1} for classical exponential dichotomies holds without an assumption of the type~\eqref{pb}. Therefore, it is natural to ask if the conclusion of Theorem~\ref{t1} is valid in the absence of~\eqref{pb}. 
However, the following example shows that the answer to this question is negative. 
\begin{example}
Let $X=\mathbb R$ with the standard Euclidean norm $\lvert \cdot \rvert$. Furthermore, let $\lVert \cdot \rVert_m=\lvert \cdot \rvert$ for $m\in \N$. We consider the sequence $(A_n)_{n\in \N}$ of operators (which can be identified with numbers) on $X$  given by
\[
A_n=\begin{cases}
n & \text{if $n=2^l$ for some $l\in \N$;}\\
0 & \text{otherwise.}
\end{cases}
\]
Let $\cA(m, n)$ be the corresponding linear cocycle. Note that $\cA(m, n)=0$ whenever $m-n\ge 2$. We claim that for each $\mathbf y=(y_n)_{n\in \N}\in Y_0$, the  sequence $\mathbf x=(x_n)_{n\in \N}$ defined by
\[
x_n=\sum_{k=1}^n \frac 1 k\cA(n, k)y_k,
\]
also belongs to $Y_0$. Indeed, observe that
\[
x_n=\begin{cases}
\frac 1 n y_n+y_{n-1} & \text{if $n=2^l+1$ for some $l\in \N$;}\\
\frac 1 n y_n & \text{otherwise.}
\end{cases}
\]
We conclude that 
\[
\lvert x_n\rvert \le \frac 1 n \lvert y_n\rvert+\lvert y_{n-1}\rvert  \quad \text{for every $n\ge 2$,}
\]
and thus
\[
\lVert x\rVert_\infty \le 2\lVert y\rVert_\infty. 
\]
However, the sequence $(A_n)_{n\in \Z}$ obviously doesn't admit a polynomial contraction since $\sup_{n\in \N}\lVert A_n\rVert=\infty$.
\end{example}

\section{Nonuniform polynomial dichotomies}\label{NPD}
In this section we  recall the notion of a nonuniform exponential dichotomy and establish its connection with the notion of a polynomial dichotomy with respect to a sequence of norms. 

We say that a sequence $(A_m)_{m\in \N}\subset B(X)$ admits a \emph{nonuniform polynomial dichotomy} if:
\begin{itemize}
\item there exist projections $P_m$, $m\in \N$ satisfying~\eqref{pro} and 
such that each map $A_m \rvert_{\Ker P_m} \colon \Ker P_m \to \Ker P_{m+1}$ is invertible;
\item there exist $\lambda, D>0$ and $\epsilon \ge 0$ such that for  $m, n\in \N$ we have 
\begin{equation}\label{nd1}
\lVert \cA(m, n)P_n \rVert \le D(m/n)^{-\lambda}n^\epsilon  \quad \text{for $m\ge n$}
\end{equation}
and 
\begin{equation}\label{nd2}
\lVert \cA(m, n)Q_n \rVert \le  D(n/m)^{-\lambda}n^\epsilon  \quad \text{for $m\le n$,}
\end{equation}
where $Q_n=\Id-P_n$ and
\[
\cA(m, n)=(\cA(n, m)\rvert_{\Ker P_m})^{-1} \colon \Ker P_n \to \Ker P_m 
\]
for $m<n$.
\end{itemize}

\begin{remark}\label{y9}
The results of Barreira and Valls~\cite{BV1}  show that the notion of a nonuniform polynomial dichotomy is quite common. Indeed, assume that $X=\R^d=\R^k \oplus \R^{d-k}$ and that each operator $A_m$ has a block-form
\[
A_m=\begin{pmatrix}
B_m & 0 \\
0 & C_m
\end{pmatrix},
\]
where $B_m \colon \R^k \to \R^k$ and $C_m \colon \R^{d-k} \to \R^{d-k}$ are linear operators. Furthermore, let $\cB(m, n)$ and $\mathcal C(m, n)$ be linear cocycles associated with those two sequences and suppose that
\[
\lim_{n\to \infty} \frac{\log \lVert \cB(n, 1)v\rVert }{\log n}<0 \quad \text{for $v\in \R^k$,}
\]
and 
\[
\lim_{n\to \infty} \frac{\log \lVert \mathcal C(n, 1)v\rVert}{\log n} >0 \quad \text{for $v\in \R^{d-k}\setminus \{0\}$.}
\]
Then, $(A_m)_{m\in \N}$ admits a nonuniform polynomial dichotomy. 
\end{remark}

\begin{proposition}\label{prop1}
The following properties are equivalent:
\begin{enumerate}
\item $(A_m)_{m\in \N}$ admits a nonuniform polynomial dichotomy;
\item $(A_m)_{m\in \N}$ admits a polynomial dichotomy with respect to a sequence of norms $\lVert \cdot \rVert_m$ satisfying
\begin{equation}\label{ln}
\lVert x\rVert \le \lVert x\rVert_m \le Cm^\epsilon \lVert x\rVert \quad x\in X, \ m\in \N
\end{equation}
for some $C>0$ and $\epsilon \ge 0$.
\end{enumerate}
\end{proposition}

\begin{proof}
Assume first that the sequence $(A_m)_{m\in \N}$ admits a nonuniform polynomial dichotomy. For each $n\in \N$ and $x\in X$, let
\[
\lVert x\rVert_n:=\sup_{m\ge n} (\lVert \cA(m, n)P_n x\rVert (m/n)^\lambda)+\sup_{m\le n}(\lVert \cA(m, n)Q_n x\rVert (n/m)^\lambda).
\]
It follows readily from~\eqref{nd1} and~\eqref{nd2} that~\eqref{ln} holds with $C=2D$. Furthermore, for $m\ge n$ and $x\in X$ we have that 
\[
\begin{split}
\lVert \cA(m, n)P_n x\rVert_m &=\sup_{k\ge m} (\lVert \cA(k, m)\cA(m, n)P_n x\rVert  (k/m)^\lambda) \\
&\le \sup_{k\ge n} (\lVert \cA(k, n)P_n x\rVert  (k/m)^\lambda) \\
&=(m/n)^{-\lambda}\sup_{k\ge n} (\lVert \cA(k, n)P_n x\rVert  (k/n)^\lambda) \\
&=(m/n)^{-\lambda} \lVert x\rVert_n,
\end{split}
\]
and thus~\eqref{d1} holds. Similarly, one can show that~\eqref{d2} holds. Therefore, $(A_m)_{m\in \N}$ admits a polynomial dichotomy with respect to the sequence of norms $\lVert \cdot \rVert_m$.

Conversely, suppose that $(A_m)_{m\in \N}$ admits a polynomial dichotomy with respect to a sequence of norms $\lVert \cdot \rVert_m$ satisfying~\eqref{ln} for some $C>0$ and $\epsilon \ge 0$. It follows that~\eqref{d1} and~\eqref{ln} that
\[
\begin{split}
\lVert \cA(m, n)P_n x\rVert  &\le \lVert \cA(m, n)P_n x\rVert_m  \\
& \le D(m/n)^{-\lambda}\lVert x\rVert_n \\
&\le CD(m/n)^{-\lambda}n^\epsilon \lVert x\rVert,
\end{split}
\]
for $m\ge n$ and $x\in X$. Therefore, \eqref{nd1} holds. Similarly, one can establish~\eqref{nd2} and therefore $(A_m)_{m\in \N}$ admits a nonuniform polynomial dichotomy.
\end{proof}
However, we will not be able to apply our main results for general nonuniform polynomial  behaviour due to the fact that the norms $\lVert \cdot \rVert_m$ constructed in the proof of Proposition~\ref{prop1} can fail to satisfy~\eqref{pb}. 
Therefore, we will consider a stronger notion of a nonuniform polynomial dichotomy. 

We say that $(A_m)_{m\in \Z}$ admits a \emph{strong nonuniform polynomial dichtotomy} if it admits a nonuniform polynomial dichotomy and there exist $K, b>0$ and $\epsilon \ge 0$ such that
\begin{equation}\label{npg}
\lVert \cA(m, n)\rVert \le K (m/n)^b n^\epsilon \quad \text{for $m\ge n$.}
\end{equation}
Observe that it is always possible to achieve that~\eqref{nd1}, \eqref{nd2} and~\eqref{npg} hold with the same $\epsilon$. 
\begin{proposition}\label{12}
The following properties are equivalent:
\begin{enumerate}
\item $(A_m)_{m\in \N}$ admits a strong  nonuniform polynomial dichotomy;
\item $(A_m)_{m\in \N}$ admits a polynomial dichotomy with respect to a sequence of norms $\lVert \cdot \rVert_m$ satisfying~\eqref{pb} and~\eqref{ln}
for some $C, M, a>0$ and $\epsilon \ge 0$.
\end{enumerate}
\end{proposition}

\begin{proof}
Assume that  $(A_m)_{m\in \N}$ admits a strong  nonuniform polynomial dichotomy. Without any loss of generality, we can suppose that $\lambda \le b$. Observe that it follows from~\eqref{nd2} (applied for $m=n$) and~\eqref{npg} that 
\begin{equation}\label{8:10}
\lVert \cA(m, n)Q_n x\rVert \le KD (m/n)^b n^{2\epsilon}\lVert x\rVert  \quad \text{for $m\ge n$ and $x\in X$.}
\end{equation}
 For $x\in X$ and $n\in \N$, set 
\[
\lVert x\rVert_n =\lVert x\rVert_n^s + \lVert x\rVert_n^u, 
\]
where
\[
\lVert x\rVert_n^s:=\sup_{m\ge n} (\lVert \cA(m, n)P_n x\rVert (m/n)^\lambda)
\]
and
\[
\lVert x\rVert_n^u:=\sup_{m\le  n}(\lVert \cA(m, n)Q_n x\rVert (n/m)^\lambda)+\sup_{m> n}(\lVert \cA(m, n)Q_n x\rVert (m/n)^{-b}).
\]
It follows readily from~\eqref{nd1}, \eqref{nd2} and~\eqref{8:10} that~\eqref{ln} for $C=(2+K)D$ and with $2\epsilon$ instead of $\epsilon$. 

On the other hand, for $m\ge n$ and $x\in X$ we have that
\[
\begin{split}
\lVert \cA(m, n)P_n x\rVert_m &=\lVert \cA(m,n)P_nx\rVert_m^s  \\
&=\sup_{k\ge m}(\lVert \cA(k,m)\cA(m,n)P_nx\rVert (k/m)^\lambda) \\
&=(m/n)^{-\lambda}\sup_{k\ge m}(\lVert \cA(k,n)P_nx\rVert (k/n)^\lambda) \\
&\le (m/n)^{-\lambda}\lVert x\rVert_n.
\end{split}
\]
Therefore, 
\begin{equation}\label{8:15}
\lVert \cA(m, n)P_n x\rVert_m \le (m/n)^{-\lambda}\lVert x\rVert_n \quad \text{for $m\ge n$ and $x\in X$.}
\end{equation}
Furthermore, for $m\le n$ and $x\in X$ we have that (using that $\lambda \le b$)
\[
\begin{split}
\lVert \cA(m, n)Q_nx\rVert_m &=\lVert \cA(m, n)Q_n x\rVert_m^u \\
&=\sup_{k\le m}(\lVert \cA(k,n)Q_n x\rVert (m/k)^\lambda)\\
&\phantom{=}+\sup_{k>m}(\lVert \cA(k, n)Q_n x\rVert (k/m)^{-b}) \\
&\le \sup_{k\le m}(\lVert \cA(k,n)Q_n x\rVert (m/k)^\lambda)\\
&\phantom{\le}+\sup_{m<k\le n}(\lVert \cA(k, n)Q_n x\rVert (k/m)^{-\lambda}) \\
&\phantom{\le}+\sup_{n<k}(\lVert \cA(k, n)Q_n x\rVert (k/m)^{-b}) \\
&\le (n/m)^{-\lambda}\sup_{k\le n}(\lVert \cA(k,n)Q_n x\rVert (n/k)^\lambda)\\
&\phantom{\le}+(n/m)^{-\lambda} \sup_{k\le n}(\lVert \cA(k, n)Q_n x\rVert (k/n)^{-\lambda}) \\
&\phantom{\le}+(n/m)^{-b}\sup_{k>n}(\lVert \cA(k, n)Q_n x\rVert (k/n)^{-b}) \\
&\le  2(n/m)^{-\lambda} \lVert x\rVert_n^u,
\end{split}
\]
and thus
\begin{equation}\label{8:16}
\lVert \cA(m, n)Q_nx\rVert_m \le 2 (n/m)^{-\lambda} \lVert x\rVert_n \quad \text{for $m\le n$ and $x\in X$.}
\end{equation}
Similarly,  one can show that 
\begin{equation}\label{8:17}
\lVert \cA(m, n)Q_nx\rVert_m \le 2(m/n)^b \lVert x\rVert_n \quad \text{for $m\ge n$ and $x\in X$.}
\end{equation}
We conclude that~\eqref{8:15} and~\eqref{8:16} imply that $(A_m)_{m\in \N}$ admits a polynomial dichotomy with respect to the sequence of norms $\lVert \cdot \rVert_m$. Furthermore, \eqref{8:15} and~\eqref{8:17} imply that~\eqref{pb} holds. 

The converse statement is straighforward to prove. 
\end{proof}

\section{Robustness of strong nonuniform polynomial dichotomy}\label{R}

\begin{theorem}\label{qe}
Let $(A_m)_{m\in \N}$ and $(B_m)_{m\in \N}$ be two sequences in $B(X)$ such that:
\begin{enumerate}
\item $(A_m)_{m\in \Z}$ admits a strong nonuniform polynomial dichotomy and let $\epsilon \ge 0$ be as in the definition of the notion of a strong nonuniform polynomial dichotomy;
\item there exist $c>0$ such that
\begin{equation}\label{robustness}
\lVert A_m-B_m \rVert \le \frac{c}{(m+1)^{2+\epsilon}} \quad \text{for $m\in \Z$.}
\end{equation}
\end{enumerate}
If $c$ is sufficiently small, then $(B_m)_{m\in \N}$ also admits a strong nonuniform polynomial dichotomy. 
\end{theorem}

\begin{proof}
Since $(A_m)_{m\in \N}$ admits a strong nonuniform polynomial dichotomy, there exists a sequence of norms $\lVert \cdot \rVert_m$ satisfying conclusions of Proposition~\ref{12}. Furthermore, it follows from Theorem~\ref{t2} that there exists a closed subspace $Z\subset X$ such that 
$T_Z \colon \mathcal D(T_Z) \to Y_0$ is an invertible operator. 

Let us consider an operator $\tilde T_Z \colon \mathcal D(T_Z) \to Y_0$ defined  by 
\[
(\tilde T_Z \mathbf x)_1=0 \quad \text{and} \quad (\tilde T_Z \mathbf x)_{m+1}=(m+1)(x_{m+1}-B_m x_m) \quad  \text{for $m\in \N$.}
\]
It follows from~\eqref{ln} and~\eqref{robustness} that 
\begin{align}\label{155}
\begin{split}
\lVert ((T_Z-\tilde T_Z)\mathbf x)_{m+1}\rVert_{m+1} &=(m+1)\lVert (A_m-B_m)x_m\rVert_{m+1} \\
&\le C(m+1)^{1+\epsilon} \lVert (A_m-B_m)x_m\rVert \\
&\le C(m+1)^{1+\epsilon} \frac{c}{(m+1)^{2+\epsilon}} \lVert x_m\rVert  \\
&=\frac{cC}{m+1} \lVert x_m\rVert \\
&\le cC\lVert x_m\rVert_m \\
&\le cC\lVert \mathbf x\rVert_\infty \\
&\le cC\lVert \mathbf x\rVert_{T_Z},
\end{split}
\end{align}
for each $m\in \N$ and $\mathbf x=(x_m)_{m\in \N}\in \mathcal D(T_Z)$.  Therefore,
\begin{equation}\label{17}
\lVert T_Z-\tilde T_Z\rVert \le cC.
\end{equation}
It follows from~\eqref{17} together with the invertibility of $T_Z$ that for $c$ sufficiently small, $\tilde T_Z$ is also an invertible operator. Hence, Theorem~\ref{t2} implies that $(B_m)_{m\in \N}$ admits a polynomial dichotomy with respect to the sequence of norms $\lVert \cdot \rVert_m$. 
Moreover, note that in~\eqref{155} we proved that 
\begin{equation}\label{66}
\lVert (A_m-B_m)x\rVert_{m+1} \le \frac{cC}{m+1}\lVert x\rVert_m \quad \text{for $m\in \N$ and $x\in X$.}
\end{equation}
Let $\cB(m, n)$ denote the linear cocycle associated with the sequence $(B_m)_{m\in \Z}$.   Furthermore, let $(y_m)_{m\in \N}\subset X$ be the sequence such that $y_{m+1}=B_my_m$ for each $m$. 
Observe that
\[
y_m=\cA(m, n)y_n+\sum_{j=n}^{m-1}\cA(m, j+1)(B_j-A_j)y_j.
\]
It follows from~\eqref{pb} and~\eqref{66} that 
\[
\lVert y_m\rVert_m \le M(m/n)^a \lVert y_n\rVert_n+McC\sum_{j=n}^{m-1} (m/j+1)^a (j+1)^{-1}\lVert y_j\rVert_j,
\]
and thus
\[
(m/n)^{-a}\lVert y_m\rVert_m \le M\lVert y_n\rVert_n+McC\sum_{j=n}^{m-1} (n/j+1)^a (j+1)^{-1}\lVert y_j\rVert_j.
\]
Using induction, it is easy to verify that 
\[
(m/n)^{-a}\lVert y_m\rVert_m \le M\lVert y_n\rVert_n \prod_{j=n}^{m-1} (1+MCc(j+1)^{-1}),
\]
and thus
\[
\begin{split}
\lVert y_m\rVert_m  
&\le M(m/n)^a \lVert y_n\rVert_n \exp \bigg{(}\sum_{j=n}^{m-1}MCc(j+1)^{-1}\bigg{)}\\
&\le M(m/n)^a \lVert y_n\rVert_n \exp \bigg{(} MCc(1+\log (m/n))\bigg{)}\\
&=Me^{MCc} (m/n)^{a+MCc}\lVert y_n\rVert_n.
\end{split}
\]
Therefore, 
\begin{equation}\label{h}
\lVert \cB(m,n)x\rVert_m \le Me^{MCc} (m/n)^{a+MCc} \lVert x\rVert_n \quad \text{for $m\ge n$ and $x\in X$.}
\end{equation}
It follows from Proposition~\ref{12} and~\eqref{h} that $(B_m)_{m\in \N}$ admits a strong nonuniform polynomial dichotomy. 
\end{proof}

\begin{remark}
A careful analysis of the proof of Theorem~\ref{qe} (see~\eqref{155}) shows that one can deduce that $(B_m)_{m\in \N}$ admits a nonuniform polynomial dichotomy under weaker assumption than~\eqref{robustness}. Indeed, it is sufficient to assume that 
\[
\lVert A_m-B_m \rVert \le \frac{c}{(m+1)^{1+\epsilon}} \quad \text{for $m\in \Z$,}
\]
with $c>0$ sufficiently small. We conclude that the more restrictive condition~\eqref{robustness} was imposed to deduce that  $(B_m)_{m\in \N}$ admits a \emph{strong} nonuniform polynomial dichotomy.
\end{remark}
We conclude by noting that our main results can be used to established parametrized robustness of nonuniform polynomial dichotomies. More precisely, one can easily obtain versions of Theorems 3 and 4 from~\cite{BDV}  in the present context. We refrein from doing so explicitly since this would require copying arguments from~\cite{BDV}.

\section{Acknowledgement}
I would like to thank anonymous referees for careful reading and for constructive comments that helped me improve the paper. 
\bibliographystyle{amsplain}

\end{document}